\documentclass{amsart}
\usepackage{euscript}           
\usepackage{eufrak} 
\usepackage{amsmath,amsthm}     
\usepackage{amssymb}            
\usepackage[usenames,dvipsnames]{color}
\usepackage{graphicx}

\newlength{\labwidth}
\newcommand{\labarrow}[1]{
\settowidth{\labwidth}{$\scriptstyle \;\; #1 \;\;$}
\stackrel{#1}{\smash{\hbox to \labwidth{\rightarrowfill}}
\vphantom{\longrightarrow}}
}

\newcommand{\sC}{{\mathcal C}}
\newcommand{\sD}{{\mathcal D}}

\newcommand{\sI}{{\mathcal I}}

\newcommand{\sS}{C\mbox{-alg}}

\newcommand{\ra}{\rightarrow}                   
\newcommand{\lra}{\longrightarrow}              
\newcommand{\la}{\leftarrow}                    

\newcommand{\holim}{\mathop{{\rm holim}}}

\newcommand{\ds}{{/\kern-3pt/}}

\newtheorem{thm}{Theorem}[section]
\newcounter{numerierer}
\newcounter{leer}

\newtheorem{defn}[thm]{Definition}

\theoremstyle{definition}  
\newenvironment{definition}{\begin{defn}\rm}{\end{defn}}

\newtheorem{example}[thm]{Example}

\usepackage[frame,matrix,curve, arrow]{xy}

\xymatrixcolsep{1.9pc}                          
\xymatrixrowsep{1.9pc}
\newdir{ >}{{}*!/-5pt/\dir{>}}                  
\newdir{ |}{{}*!/-5pt/\dir{|}}                  

\subjclass{}
\begin{document}

\title{A note on the realization of relative $h_\infty$-diagrams}
\author{Gerd Laures}

\address{ Fakult\"at f\"ur Mathematik,  Ruhr-Universit\"at Bochum, NA1/66, D-44780 Bochum, Germany}
\setcounter{section}{1}

\subjclass[2010]{Primary 55N34; Secondary 55P50, 57R20}

\date{\today}

\begin{abstract}
We prove a relative version of the realization theorem for $h_\infty$-diagrams in case that the underlying diagram subcategory is factorization-closed.
\end{abstract}

\maketitle

The rigidification of diagrams which only commute up to homotopy has been investigated  by various authors.  The most striking result goes back to Dwyer-Kan-Smith-Stover  \cite{MR984042} and Devinatz-Hopkins \cite[Theorem 3.2]{MR2030586}. They  proved  realizability in the case that  all mapping spaces have contractible path components.  In  this brief note we prove a relative version of this theorem.  The source subcategory should have the same objects and its morphisms should be factorization-closed. This allows us to neglect morphisms for which the path components are not contractible but for which the diagram already strictly commutes. We will show that the explicit proof given in \cite[Theorem 3.2]{MR2030586} carries over from the absolute to the relative situation.

\begin{definition}
A subcategory ${\sC}$ of a category  ${\sD}$ is called {\em factorization-closed} if all composable morphisms $\alpha, \beta \in {\sD}$ satisfy
$$ \alpha \beta  \in {\sC} \Longrightarrow  \alpha   \in {\sC} \mbox{ and }  \beta  \in {\sC}.$$
\end{definition}
\begin{example} An important example looks like
\bigskip
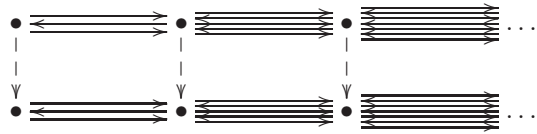
\begin{center}
$ \xymatrix{
\bullet  \ar@<3pt>[rr] \ar@<-3pt>[rr]\ar@{-->}[d]&&
\bullet    \ar@<4pt>[rr]\ar@<-4pt>[rr]\ar[rr]\ar@<0pt>[ll]\ar@{-->}[d]&&
\bullet  \ar@<6pt>[rr]\ar@<-6pt>[rr]\ar@<2pt>[rr]\ar@<-2pt>[rr]
\ar@<2pt>[ll]\ar@<-2pt>[ll]\ar@{-->}[d]
&&\ldots \ar@<0pt>[ll]\ar@<-4pt>[ll]\ar@<4pt>[ll]\\
\bullet  \ar@<3pt>[rr]\ar@<-3pt>[rr]&&
\bullet    \ar@<4pt>[rr]\ar@<-4pt>[rr]\ar[rr]\ar@<0pt>[ll]&&
\bullet  \ar@<6pt>[rr]\ar@<-6pt>[rr]\ar@<2pt>[rr]\ar@<-2pt>[rr]
\ar@<2pt>[ll]\ar@<-2pt>[ll]
&&\ldots \ar@<0pt>[ll]\ar@<-4pt>[ll]\ar@<4pt>[ll]
}$
\end{center}
\bigskip
where the dashed arrows do not belong to $\sC$. That is, 
 $ {\sC}$ consists of two copies of the simplicial category $ \Delta $ and ${\sD}$ is the category which classifies morphisms between cosimplicial objects. 
\end{example}
Suppose $\sC$ is a subcategory of a small category $\sD$ with the same objects. For the target category of our diagrams,  let $C$ be a good choice of an $A_\infty$-operad. More precisely,  assume that the augmented simplicial spectrum $C^{\bullet +1}E$ is Reedy cofibrant for all cofibrant $C$-algebras $E$. Write ${\sS}$ for the category of $A_\infty$-ring spectra.  We will study the problem of finding a {\em realization} of a  diagram $(X,\tilde{Y})$ of the form 
$$ \xymatrix{\sC \ar@{>->}[d]\ar[r]^X & \sS \ar[d]^\pi \\ 
\sD  \ar[r]^{\tilde{Y}} & Ho( \sS )}$$
This means we are looking for  functors $Y:\sD\ra \sS$ and $X_i: \sC \ra \sS$, $1\leq i\leq  n$ and a zigzag of weak equivalences
$$\xymatrix{ Y\mid_{\sC} \ar[r]^{\varphi_1}  & X_1 &\ar[l]_{\varphi_2} X_2 \ar[r]^{\varphi_3}&  \cdots &  X_n \ar[l]_{\varphi_n}  \ar[r]^{\varphi_{n+1}} & X}$$
with the property that 
$\varphi_{n+1}\cdots \varphi^{-1}_{2}\varphi_1: \pi Y \ra \bar{Y}$ is a natural equivalence. Here,  the last arrow will point to the left if $n$ is odd.
\begin{definition}
Let $Z=(X,\bar{Y})$  be a diagram as above. We say that $Z$ is an $h_\infty$-diagram if for each morphism  $\alpha: j_1 \ra j_2$ in $\sD$ which is not contained in $\sC$ the space $\sS(Xj_1,Xj_2)_{\bar{Y}\alpha}$ is contractible. Here,  $\sS(Xj_1,Xj_2)_{\bar{Y}\alpha}$ is the path component of $\sS(Xj_1,Xj_2)_{\bar{Y}\alpha}$ which contains $\bar{Y}\alpha$.
\end{definition}
\begin{thm}\label{h infty}
Suppose ${\sC}$ is a factorization-closed subcategory of $\sD$. Then every  $h_\infty$-diagram is realizable.
\end{thm}
\begin{proof}
We adopt the notation of \cite{MR2030586} and write $\Gamma T$ for the realization of the singular simplicial set of an unbased topological space $T$.

For a morphism  $\alpha: j \ra j'$ in $\sD$ set
$$ M(X_j,X_{j'})_\alpha =\left\{  \begin{array}{ll}
\{ X\alpha \} & \mbox{ if } \alpha \in \sC\\
\Gamma \sS(X_j,X_{j'})_{\bar{Y}\alpha} & \mbox{ else } 
\end{array} \right.
$$
Let $\alpha: j' \ra j''$ be another morphism.  Then the composition of maps induces a pairing
$$\circ: M(X_{j'},X_{j''})_{\alpha'} \times  M(X_j,X_{j'})_\alpha  \lra M(X_j,X_{j''})_{\alpha' \alpha} $$
which is associative since ${\sC}$ is factorization-closed.
Define a cosimplicial $C$-algebra $\Pi_h^*Z$ by
$$\Pi_h^0Z=\Pi_{j\in {\sC}} Xj; \qquad \Pi^n_h Z=\Pi_{\sD_n}F(M\alpha, Xj_0)$$
where $\sD_n$ is the set of diagrams
$$\xymatrix{  \alpha: \: j_0 & j_1\ar[l]_{\alpha_1} & j_2 \ar[l]_{\alpha_2} & \cdots \ar[l] & j_n\ar[l]_{\alpha_n}}$$
in $\sD$, 
$$M\alpha =  M(X_{j_0},X_{j_1})_{\alpha_1}\times \cdots \times M(X_{j_n},X_{j_{n-1}})_{\alpha_n}$$
and $F(T,X)$ is the function spectrum, that is, the cotensor product of the space $T$ with the spectrum $X$. The cofaces $d^i$ are induced by the composition $\circ$ and are defined in exactly the same way as in \cite[Construction 3.3]{MR2030586}:
For $0<i<n+1$ the coface $d^i: \Pi_h^nZ \ra \Pi_h^{n+1}Z$ is defined in the factor indexed by $\alpha: j_0\la \cdots \la j_{n+1}$ by the composite
$$ \Pi_h^nZ \stackrel{\pi_{\alpha'}}{\ra} F(M\alpha',Xj_0) \stackrel{d^i_\alpha}{\ra} F(M\alpha, Xj_0)$$
where $$\alpha': j_0\stackrel{\alpha_1}{\la}j_1 \la \cdots \la j_{i-1} \stackrel{\alpha_i\alpha_{i+1}}{\la}j_{i+1} \stackrel{\alpha_{i+2}}{\la} \cdots \stackrel{\alpha_{n+1}}{\la} j_{n+1}$$
and
  $$ (d_\alpha^i g)(f_1,\ldots, f_{n+1})=g(f_1, \ldots, f_i\circ f_{i+1},\ldots, f_{n+1}).$$  
For $i=0$ it is defined in the same way with
$$\alpha': j_1\stackrel{\alpha_2}{\la}j_2 \la \cdots \stackrel{\alpha_{n+1}}{\la} j_{n+1}$$ and 
 $$ (d_\alpha^0 g)(f_1,\ldots, f_{n+1})=f_1(g(f_2, \ldots, f_{n+1})).$$
Finally, for $i=n+1$ set
$$\alpha': j_0\stackrel{\alpha_1}{\la}j_1 \la \cdots \stackrel{\alpha_{n}}{\la} j_{n}$$   and
 $$ (d_\alpha^{n+1} g)(f_1,\ldots, f_{n+1})=g(f_1, \ldots, f_{n})).$$
 Associativity ensures that the cosimplicial identities hold. The codegeneracies are defined via the evaluations on identity maps. Then \cite[Lemma 3.6]{MR2030586} shows that $\Pi_h^*Z$ is fibrant.

For an object $j$ of $\sD$ let $\sD \backslash  j$ be the under category of $j$. Its objects are morphisms from $j$ to some object $j'$ and its morphisms are commutative triangles
$$\xymatrix{ &j\ar[ld]\ar[rd]& \\ j'\ar[rr]^\alpha&&j''}.$$
Let $(\sD\backslash j )_\sC$ be  the subcategory of $\sD \backslash  j$ with the same objects and with morphisms $\alpha \in \sC$. Note that $(\sD\backslash j )_\sC$ is factorization-closed.
Let $\mu_j: D\backslash j \lra \sD$ be the evident forgetful functor which maps the subcategory  $(\sD\backslash j )_\sC$ to $\sC$. The $h_\infty$-diagram $Z=(X,\bar{Y})$ provides us with an $h_\infty$-diagram $\mu^*_j Z= (\mu^*_jX,\mu^*_j\bar{Y})$ over the pair $(\sD\backslash j ,(\sD\backslash j )_\sC)$. 
Set
$$ Y(j)=\mbox{Tot} (\Pi_h^*\mu^*_jZ).$$
Here, the totatlization Tot$(W)$ of a cosimplicial spectrum $W$ is the spectrum $F(\Delta[*] , W)$ of cosimplicial maps from the standard cosimplicial space $\Delta[*]$ to $W$. 
$Y$ is a functor from $\sD$ to $\sS$ in the obvious way:  a morphism $f:j\ra j'$ gives a functor $f^* : D\backslash j'\ra  D\backslash j $ compatible with the subcategories. This functor induces a map on the cosimplicial spaces and hence on their totalizations. We claim that $Y$ is the desired realization. \par
As in the absolute case there is a map $Y(j)\ra \bar{Y}(j)$ given by the projection
$$p_j: \mbox{Tot}(\Pi_h^*\mu^*_jZ)\lra F(\Delta[0],\Pi_h^0\mu^*_jZ)=\Pi_h^0\mu^*_jZ\stackrel{pr}{\lra} Xj$$
onto the factor indexed by the identity of $j$. The maps are natural in the homotopy category for the same reason as in the absolute case: the cosimplicial maps of degree one provide the homotopies  which make the naturality diagram commute. We claim that the maps $p_j$ are weak equivalences.
The Bousfield-Kan spectral sequence \cite[X.6,7]{MR0365573}[X.6,7] takes the form
$$ \pi_{t-s}(\mbox{Tot} (\Pi_h^*\mu^*_jZ)\Longleftarrow E_2^{s,t}\cong \pi^s (\pi_t(\Pi_h^*\mu^*_jZ)) $$
Since $Z$ is an $h_\infty$-diagram each $M_\alpha$ is contractible and hence
 $$\pi^s (\pi_t(\Pi_h^*\mu^*_jZ))\cong \lim_{\sD \backslash j} {}^{\! s}\pi_t(\mu^*_jZ).$$
 Finally, since  the identity map of $j$ is initial in $\sD \backslash j$ we have that the higher limits vanish and for $s=0$ it coincides with $\pi_tXj$. \par
 It remains to analyze the restriction of $Y$ to $\sC$. 
  Let $\Pi^*$ be the standard cosimplicial  replacement (cf. \cite[XI5.1]{MR0365573}) of a  diagram  $F$ on a small category $\sI$ given in codimension $n$ by
$$ \Pi^nF= \prod_{\sI_n} Fi_0 .$$
 There is an obvious cosimplicial map $$ \Pi^*_h\mu_j^*Z \lra \Pi^*\mu_j^*X$$
induced by projections onto the given factors since all $M_\alpha$ are points for morphisms in $\sC$.
 This yields  an equivalence  on  totalizations
 $$ Y_j \lra \mbox{Tot} (\Pi^*\mu_j^*X )= \holim_{\sC \backslash j} \mu^*_jX$$
which is natural in $\sC$. The natural map
$$ X_j \cong  \lim_{\sC \backslash j} \mu^*_jX \lra  \holim_{\sC \backslash j} \mu^*_jX$$
is again an equivalence by the Bousfield-Kan spectral sequence. This completes the desired zig-zag.
\end{proof}
\medskip
\bibliography{toda}
\bibliographystyle{amsalpha}

\end{document}